\documentclass{amsart}
\usepackage[utf8]{inputenc}
\usepackage{amsmath}
\usepackage{amssymb}
\usepackage{amsfonts}
\usepackage{cleveref}
\usepackage{tikz-cd}
\usepackage{stmaryrd}

\usepackage{MnSymbol}
\DeclareMathAlphabet\mathbb{U}{msb}{m}{n}

\newtheorem{theorem}{Theorem}[section]
\newtheorem{corollary}[theorem]{Corollary}
\newtheorem{lemma}[theorem]{Lemma}
\newtheorem{proposition}[theorem]{Proposition}
\theoremstyle{definition}
\newtheorem{definition}[theorem]{Definition}

\def\Z{\mathbb{Z}}

\def\L{\mathfrak{l}}
\def\g{\mathfrak{g}}
\def\h{\mathfrak{h}}
\def\a{\mathfrak{a}}
\def\b{\mathfrak{b}}
\def\f{\mathfrak{f}}
\def\Wedge{\,\Tilde{\wedge}\,}
\def\Ker{{\rm Ker}}
\def\k{K}
\def\lp{(\!(}
\def\rp{)\!)}
\def\lb{\llbracket}
\def\rb{\rrbracket}

\begin{document}

\title{Homological properties of parafree Lie algebras}

\begin{abstract}
 In this paper, an explicit construction of a countable parafree Lie algebra with nonzero second homology is given. It is also shown that the cohomological dimension of the pronilpotent completion of a free noncyclic finitely generated Lie algebra over $\mathbb Z$ is greater than two. Moreover, it is proven that there exists a countable parafree group with nontrivial $H_2.$
\end{abstract}

\author{Sergei O. Ivanov}
\address{
Laboratory of Modern Algebra and Applications,  St. Petersburg State University, 14th Line, 29b,
Saint Petersburg, 199178 Russia}\email{ivanov.s.o.1986@gmail.com}

\author{Roman Mikhailov}
\address{Laboratory of Modern Algebra and Applications, St. Petersburg State University, 14th Line, 29b,
Saint Petersburg, 199178 Russia and St. Petersburg Department of
Steklov Mathematical Institute} \email{rmikhailov@mail.ru}

\author{Anatolii Zaikovskii}
\address{
Laboratory of Modern Algebra and Applications,  St. Petersburg State University, 14th Line, 29b,
Saint Petersburg, 199178 Russia}\email{anat097@mail.ru}

\thanks{The authors are supported by the Russian Science Foundation grant N 16-11-10073.}

\maketitle

\section{Introduction}
\subsection*{Historical remarks.} In 1960-s G. Baumslag introduced the class of parafree groups \cite{Baumslag1}, \cite{Baumslag2}, \cite{Baumslag3}, \cite{Baumslag4}. Recall that a group $G$ is called {\it parafree} if $G$ is residually nilpotent and there exists a free group $F$ and a homomorphism $F\to G$, which induces isomorphisms of the lower central quotients $F/\gamma_i(F)\simeq G/\gamma_i(G),\ i\geq 1$. The main motivation to introduce and study parafree groups was the problem how to characterize the class of groups of cohomological dimension one. G. Baumslag called the parafree groups as "just about free" and had a hope that some of them may give examples of non-free groups of cohomological dimension one. At the end of 1960-s the results of J. Stallings and R. Swan appeared \cite{Stallings2}, \cite{Swan}. By Stallings-Swan theorem, the groups of cohomological dimension one are free, hence, the non-free parafree groups constructed in \cite{Baumslag1}, \cite{Baumslag2} have cohomological dimension at least two. Despite this fact, there are series of properties which parafree groups share with free groups. G. Baumslag during many years studied these properties and stated a number of natural problems about parafree groups. The main conjecture about homological properties of parafree groups is known as {\it Parafree Conjecture}: {\it for a finitely generated parafree group $G$, $H_2(G, \mathbb Z)=0$.} There is an additional strong form of the conjecture: {\it for a finitely generated parafree group $G$, $H_2(G, \mathbb Z)=0$ and the cohomological dimension of $G$ is $\leq 2$.} For the formulation of these conjectures we refer to \cite{Cochran}, see also \cite{CO} for the discussion of topological applications of parafree groups. T. Cochran wrote the following: "Some (including Baumslag) believe that all finitely-generated parafree groups have cohomological dimension at most 2 and have trivial $H_2$." The fact that in \cite{Cochran} the Parafree Conjecture is formulated for finitely generated groups only is due to the result of A.K. Bousfield \cite{B1}: {\it for a non-cyclic finitely generated free group $F$, and its pronilpotent completion $\hat F:={\sf lim}\ F/\gamma_i(F)$, $H_2(\hat F,\Z)$ is uncountable.} The pronilpotent completion $\hat F$ is parafree, hence it gives an example of a non-free parafree group with $H_2\neq 0.$ Observe that, the initial interest of G. Baumslag was not in just finitely generated parafree groups, but in parafree groups in general, in particular, he constructed locally free non-free parafree groups in \cite{Baumslag4}, and the problem of sharing the properties of parafree groups with free groups first was formulated in general, not only for finitely generated case.

In 2010-2014 the third author worked with G. Baumslag on constructing examples of countable (non-finitely generated, in general) parafree groups with nonzero $H_2$ and (or) cohomological dimension greater than two. That project was not finished and, for the moment, we are not able to present a (probably uncountable) parafree group of cohomological dimension greater than two. The problem whether the cohomological length of $\hat F$ is greater than two is still open. Next we will show that there exist countable parafree groups with $H_2\neq 0$. However, an explicit construction of such groups seems problematic. In this paper we make a step in this direction, by constructing explicit examples of countable parafree Lie algebras with $H_2\neq 0$ as well as an example of a parafree Lie algebra of cohomological length greater than two.

\subsection*{Countable parafree groups with nonzero $H_2$.}

If $F$ is a finitely generated free group of rank at least $2$, we say that $G$ is a parafree subgroup of $\hat F,$ if $F\subseteq G\subseteq \hat F $ and the embedding $F\subseteq G $ induces isomorphisms $F/\gamma_i(F)\cong G/\gamma_i(G),\ i\geq 1.$ It is easy to prove that there exists a countable parafree group  with non-trivial $H_2$ using the following proposition.

\medskip

\noindent {\bf Proposition 1.} {\it $\hat F$ is a filtered union of its countable parafree subgroups.}

\medskip

\noindent This proposition implies that $\hat F={\sf colim}\: G,$ where $G$ runs over the directed set of countable parafree subgroups of $\hat F$.  Since $H_2(-, \mathbb Z)$ commutes with filtered colimits, we obtain ${\sf colim}\: H_2(G, \mathbb Z) = H_2(\hat F, \mathbb Z) \ne 0.$ Therefore, there exists $G$ such that $H_2(G, \mathbb Z)\ne 0.$ Moreover, using that $H_2(\hat F, \mathbb Z)$ is uncountable, we obtain the following.

\medskip

\noindent {\bf Corollary.} {\it There exists an uncountable set of countable parafree subgroups $G\subset \hat F$ such that $H_2(G, \mathbb Z)\ne 0.$}

\subsection*{Parafree Lie algebras.} The concept of parafreeness can be naturally extended from groups to other algebraic categories, such as Lie algebras or augmented associative algebras. The following question rises naturally: {\it what kind of properties do free and parafree objects have in common}? Can one construct a finitely generated but not finitely presented parafree object? What can one say about homology and cohomological dimension of parafree objects?

Parafree Lie algebras are considered in \cite{BS}.
Suppose that $R$ is a commutative associative ring. For a Lie algebra $\g$ over $R$ we denote by $\gamma_i(\g)$ the lower central series of $\g.$ We also denote by $\gamma_\omega( \g)$ the intersection :
$$\gamma_\omega(\g)  = \bigcap\limits_{i \geq 1} \gamma_i(\g).$$ A Lie algebra $\g$ is called parafree if $\gamma_\omega(\g)=0$ and there exists a free Lie algebra $\mathfrak{f}$ (over R) with a homomorphism $\mathfrak{f} \to \g$ which induces isomorphisms $\mathfrak{f}/\gamma_i (\mathfrak{f})\simeq \g/\gamma_i(\g)$ for all $i=1,2,\dots$

For a Lie algebra $\g$, the {\it pronilpotent  completion} $\hat{\g}$ is the inverse limit $\varprojlim \g/\gamma_i(\g).$ It follows immediately from definition that, for a parafree Lie algebra $\g$ there is an isomorphism $\hat{\mathfrak{f}}\cong \hat{\g},$ where $\mathfrak{f}$ is a free Lie algebra.

The main results of this paper are Theorems A and B, formulated bellow. Denote by $\k$ a field of characteristic $2$ and by $\b$ the Lie algebra over $\k$ generated by elements $a, b, \{x_i\}, \{y_i\}$ for $i \geqslant 1$ with the following relations\footnote{For elements of Lie algebras we will use the left-normalized notation
$[a_1,\dots, a_i]:= [[a_1,\dots, a_{i-1}],a_i]$
and the following notation for Engel commutators
$[a,_0 b] := a$ and $[a,_{i+1} b] := [[a,_i b], b]$
for $i \geqslant 0.$}:
\begin{align*}
& x_1 = [a, b, b] + [x_2, b, b], \ \ \ y_1 = [a, b, a] + [y_2, b, b],\\
& \ldots\\
& x_i = [a, b, b] + [x_{i+1}, {}_{2^i} b], \ \ \ y_i = [a, b, a] + [y_{i+1}, {}_{2^i} b],\ {\text for\ all}\ i\geq 1.
\end{align*}

\medskip 

\noindent{\bf Theorem A.} {\it The Lie algebra $\mathfrak{a}= \b/\gamma_\omega(\b)$ is parafree and $H_2(\mathfrak{a},  \k)\neq 0.$}

\medskip

\noindent{\bf Theorem B.} {\it  Let $\f$ be a free Lie algebra over $\mathbb Z$ of rank two. The homology group of the pronilpotent completion  $H_2(\hat{\f}, \mathbb Z)$ contains a 2-divisible element. Hence, the cohomological dimension of $\hat \f$ is greater than two.}

\medskip

The proofs are based on the method used in the solution of Bousfield's problem \cite{IM}, \cite{IM2}. All results in \cite{IM}, \cite{IM2} are for groups. Here we prove their analogs for Lie algebras. In particular, we introduce the Lie-analog $\L_R$ of the lamplighter group. We show that certain elements of the second homology $H_2(\hat \f_R, R)$ are nonzero by projecting them onto the elements of $H_2(\hat{\L}_R,R)$. The main tool for showing that an element in $H_2(\hat{\L}_R,R)$ is nonzero is given by Corollary \ref{rational}: non-triviality of an element in homology follows from the non-rationality of certain formal power series.

We hope that the results of this paper will help to attack the homological problems for groups. In particular, the proof of ${\sf cd}(\hat{\mathfrak{f}} )>2$ in the case $R=\Z$ gives an approach for the proof of ${\sf cd}(\hat F_\Z)>2$, however, the group case is more complicated. Generally speaking, the theories of parafree Lie algebras and parafree groups are very similar, but the theory of
parafree augmented associative algebras is different from groups or Lie algebras. The pronilpotent completion of the free augmented associative algebra has cohomological dimension one but contains subobjects of cohomological dimension $\geq 2$, what is not possible in the category of groups. In view of this difference, it is not surprising that, a finitely presented parafree augmented associative algebra of cohomological dimension $>2$ can be constructed. The authors hope to give such kind of examples in the forthcoming papers.

\section{Proofs}
Throughout the paper we denote by $R$ a commutative associative ring. For a given set $X$ we denote by $\f = \f_R(X)$ the free Lie algebra over $R$ generated by a set $X.$ The free Lie algebra has a natural grading $\f = \oplus_{n \geqslant 1} \f_n.$  Note that an element $x$ from $\hat{\f}$ can be treated as an infinite series $x = \sum x_n$ with $x_n \in \f_n.$

For a Lie algebra $\mathfrak{g}$ over $R$ the $i$-th homology $H_i( \mathfrak{g} , -) = Tor^{U( \mathfrak{g} )}_i(R, -),$ where $U(\mathfrak{g})$ is the universal enveloping algebra and $R$ is viewed as the trivial module over $U(\mathfrak{g}).$ In the proofs we will often use description of second homology group in terms of non-abelian exterior square. That description was firstly found for groups by Miller \cite{Miller} and then it was found for Lie algebras by Ellis \cite{Ellis}.

\begin{definition}[\cite{Ellis}]
For a Lie algebra $\g$ over $R$ the exterior square $\g \Wedge \g$ is a Lie algebra, which underlying $R$-module is ${\rm Coker}(\varphi : \g \wedge_R \g \wedge_R \g \to \g \wedge_R \g),$ where $\varphi$ is given by 
    $$\varphi(m \wedge n \wedge k) = [m, n] \wedge k + [n, k] \wedge m + [k, m] \wedge n,$$
and a Lie bracket is given by 
    $$[m \wedge n, m' \wedge n'] = [m, n] \wedge [m', n'].$$
For $m \wedge n \in \g \wedge_R \g$ we denote by $m \Wedge n$ its representative in $\g \Wedge \g.$
\end{definition}

There is a canonical map $\g \Wedge \g \to \g$ given by $m \Wedge n \mapsto [m, n].$

\begin{proposition}[\cite{Ellis}]\label{Ellis description}
Suppose that $\g$ is a Lie algebra over $R.$ Then there is a natural isomorphism 
$$H_2(\g, R) \cong \Ker(\g \Wedge \g \to \g).$$
\end{proposition}

\begin{corollary}
If $\g$ is an abelian Lie algebra, then $H_2(\g, R) \cong \g \Wedge \g \cong \g \wedge_R \g.$
\end{corollary}

\subsection{Free Lie algebra pronilpotent completion homology}

It was shown in \cite{B1} and later in \cite{IM} that $H_2(\hat{F}_{\mathbb Z}, {\mathbb Z})$ is uncountable, where $\hat{F}_{\mathbb Z}$ is the pronilpotent completion of the free group on two generators. Here we provide an analogue of this result for the free Lie algebra $\f_R$ of rank two using methods from \cite{IM} and \cite{IM2}.

\begin{definition}
By $\L_R$ we denote the semidirect sum of abelian Lie algebras $R[x] \rtimes Q,$ where $Q = Rt$ is a free $R$-module of rank $1$ generated by an element $t.$ The generator $t \in Q$ acts on a polynomial $p \in R[x]$ as follows:
    $$[p, t] = p x.$$
\end{definition}

For a commutative ring $R$ we denote the ring of formal power series over $R$ by $R\lb x\rb.$ We consider it as an abelian Lie algebra.
\begin{lemma}
The pronilpotent completion $\hat{\L}_R$ is a semidirect sum $R\lb x\rb \rtimes Q$ with the action 
    $$[f, t] = f x$$
for $f \in R\lb x\rb$ and the generator $t \in Q.$
\end{lemma}
\begin{proof}
Since $\gamma_n(\L_R) = x^n R[x]$ and $\L_R / \gamma_n(\L_R) = R[x]/\left(x^n\right) \rtimes Q$ with similar action, the assertion follows from the fact that $\varprojlim R[x]/(x^n) = R\lb x\rb.$
\end{proof}

\begin{lemma}\label{Lamplighter homology}
There is an isomorphism
    $$(R\lb x\rb\wedge_R R\lb x\rb)_Q \cong H_2(\hat{\L}_R, R)$$
induced by the inclusion $R\lb x\rb \to \hat{\L}_R,$ where the action $Q$ on $R\lb x\rb\wedge_R R\lb x\rb$ is given by
    $$(f \wedge g) \cdot t = fx \wedge g + f \wedge gx.$$
\end{lemma}
\begin{proof}
Consider the short exact sequence $R\lb x\rb \rightarrowtail \hat{\L}_R \twoheadrightarrow Q$ and the associated Lyndon–Hochschild–Serre spectral sequence 
    $$E_{i, j}^2 = H_i(Q, H_j(R\lb x\rb, R)) \implies H_{i + j}(\hat{\L}_R, R),$$
where $R$ is the trivial module over $\hat{\L}_R.$ Since $Q$ is a free Lie algebra of rank one, we have $H_n(Q, -) = 0$ for $n \geqslant 2.$ It follows that $E^2_{i, j} = 0$ for $i \geqslant 2,$ and hence, there is a short exact sequence
    $$0 \longrightarrow E^2_{0,2} \longrightarrow H_2(\hat{\L}_R, R) \longrightarrow E^2_{1,1} \longrightarrow 0.$$
The action of $Q$ on $R\lb x\rb$ has no invariants, so $E^2_{1,1} = H_1(Q, R\lb x\rb) \cong R\lb x\rb^Q = 0,$ and the inclusion
    $$H_2(R\lb x\rb, R)_Q = E^2_{0, 2} \longrightarrow H_2(\hat{\L}_R,R)$$
is an isomorphism. Since $R\lb x\rb$ is abelian, $H_2(R\lb x\rb, R) \cong R\lb x\rb\wedge_R R\lb x\rb$ and the assertion follows. 
\end{proof}

\begin{lemma}\label{Coinvariants description}
Suppose that $\sigma: R\lb x\rb \to R\lb x\rb$ is an involution given by
    $$\sigma(\sum\limits_{i \geqslant 0} q_i x^i) = \sum\limits_{i \geqslant 0} (-1)^i q_i x^i.$$
Then there is an isomorphism
    $$\eta: (R\lb x\rb\wedge_R R\lb x\rb)_Q \xrightarrow{\sim} R\lb x\rb \wedge_{R[x]} R\lb x\rb$$
given by $\eta(\overline{f \wedge g}) = f \wedge \sigma(g),$ $\eta^{-1}(f \wedge g) = \overline{f \wedge \sigma(g)}.$
\end{lemma}
\begin{proof}
Using a property of sigma $\sigma(gx) = -x\sigma(g)$ it is easy to check, that the maps above are well-defined. Also those maps are obviously inverse to each other.
\end{proof}

\begin{lemma}\label{rational}
Suppose that $R$ is an integral domain and  $p \in R\lb x\rb.$ If $p \wedge 1 = 0 \in R\lb x\rb \wedge_{R[x]} R\lb x\rb$, then $p$ is a rational function over $R.$
\end{lemma}
\begin{proof}
Let us denote by $R(x)$ the field of rational functions over $R$ and denote by $R\lp x\rp = R\lb x\rb \otimes_{R[x]} R(x)$ the field of Laurent power series. Then the following diagram of $R[x]$ modules is commutative 
$$
\begin{tikzcd} 
    R\lb x\rb \arrow[rr] \arrow[d, rightarrowtail] && R\lb x\rb \wedge_{R[x]} R\lb x\rb \arrow[d] \\
    R\lp x\rp \arrow[rr] && R\lp x\rp \wedge_{R(x)} R\lp x\rp,
\end{tikzcd}
$$
where horizontal maps send $p$ to $p \wedge 1.$
Since $R\lb x\rb$ is an integral domain we can regard it as a subring of $R\lp x\rp.$ Hence 
    $$\Ker\left(R\lb x\rb \to R\lb x\rb \wedge_{R[x]} R\lb x\rb\right) \subseteq \Ker\left(R\lp x\rp \to R\lp x\rp \wedge_{R(x)} R\lp x\rp\right)$$
and it is sufficient to show that if $p$ is in the right set, then it is in $R(x).$ Indeed, if $p \wedge 1 = 0$ for $p \in R\lp x\rp,$ then $p$ and $1$ are linearly
dependent over $R(x),$ so $p \in R(x).$
\end{proof}

\begin{lemma}\label{lemma_not_rational}
Let $R$ be an integral domain and $p=\sum_{i=0}^\infty a_ix^i \in R\lb  x\rb$ is a power series, which is not a polynomial, and such that for any $n\geq 1$ there exists $k\geq 0$ such that $a_k\ne 0$ and $a_{k+1}=\dots =a_{k+n}=0.$ Then $p$ is not rational. 
\end{lemma}
\begin{proof} Consider a polynomial $q=\sum_{i=0}^n b_i x^i.$ Since $p$ is not a polynomial, there are infinitely many numbers $k$ such that $a_k\ne 0$ and $a_{k+1}=\dots =a_{k+n}=0.$ Then for any such $k$ the $k+n$-th coefficient of $p\cdot q$ is $a_{k}b_{n}.$ Therefore $p\cdot q$ is not a polynomial.  
\end{proof}

For a sequence $\alpha=(\alpha_n)\in \{0,1\}^{\mathbb N}$ we consider the following power series 
$$p_\alpha =\sum_{n=1}^\infty \alpha_nx^{2^n}$$

\begin{lemma}\label{lemma_image_uncount}
Let $R$ be an integral domain. Then the image of the map 
$$R\lb x\rb \to R\lb x\rb \wedge_{R[x]} R\lb x\rb$$ given by $p \mapsto p \wedge 1$ is uncountable. Moreover the set 
$$\{p_\alpha \wedge 1 \in R\lb x\rb \wedge_{R[x]} R\lb x\rb \mid \alpha \in \{0,1\}^{\mathbb N} \}$$
is uncountable.
\end{lemma}
\begin{proof} Consider an equivalence relation on the set of all sequences  $\{0,1\}^{\mathbb N}$ such that $(\alpha_n) \sim (\beta_n)$ if and only if the set $\{n\mid \alpha_n\ne \beta_n\}$ is finite. All equivalence classes of this equivalence relation are countable. Hence the quotient set $\{0,1\}^{\mathbb N}/\sim $ is uncountable. Take a set $A\subseteq \{0,1\}^{\mathbb N}$ which intersects with any equivalence class by one element. Then $A$ is uncountable.   \Cref{lemma_not_rational} implies that $p_\alpha-p_\beta$ is not rational for any distinct $\alpha,\beta\in A.$ \Cref{rational} implies that $p_\alpha\wedge 1 \ne p_\beta \wedge 1 \in R\lb x\rb \wedge_{R[x]} R\lb x\rb.$ The assertion follows. 
\end{proof}

Let $R$ be an integral domain and  $\f_R = \f_{R}(a, b)$ be the free Lie algebra over $R$ of rank two. Denote by 
$$\varphi: \f_R\to \L_R $$
the Lie algebra homomorphism which is defined on generators  by $a \mapsto 1_{R\lb x\rb}$ and $b \mapsto t.$ Then $\varphi$ induces a map $H_2(\hat \f_R,R)\to H_2(\hat \L_R,R).$

\begin{theorem}\label{image}
The image of the map $$H_2(\hat{\f}_R, R) \to H_2(\hat{\L}_{R}, R)$$ is
uncountable.
\end{theorem}
\begin{proof}
Set $\f=\f_R$ and $\L=\L_R.$ For a Lie algebra $\g$ we regard its second homology group as $\Ker(\g \Wedge \g \to \g)$ using \Cref{Ellis description}. Then the induced on homology map is $\hat{\varphi} \Wedge \hat{\varphi},$ where $\hat{\varphi}: \hat{\f} \to \hat{\L}$ is induced by $\varphi.$ 
Denote by $r_{2n}, s_{2n}$ the following elements from $\f$
    $$r_{2n} = [a, {}_{2n}b], \ \ \  s_{2n} = \left(\sum\limits^{n}_{i=1} (-1)^i[[a, {}_{2n-i}b], [a, {}_{i-1}b]] \right),$$
and for any $\alpha \in \{0, 1\}^{\mathbb N}$ denote by $R_\alpha, S_\alpha$ the following elements from $\hat{\f}$
    $$R_\alpha = \sum_{n \in \mathbb N} \alpha_n r_{2n}, \ \ \   S_\alpha = \sum_{n \in \mathbb N} \alpha_n s_{2n}.$$
It was shown in \cite[Lemma 4.1]{IM} that for any Lie algebra $\mathfrak{g}$ the following identities hold for all $x, y \in \mathfrak{g}$ and for $n \geqslant 1:$
    $$[[x, {}_{2n}y], x] = \left[\sum\limits^{n-1}_{i=0} (-1)^i[[x, {}_{2n-1-i}y], [x, {}_i y]], y\right].$$
Therefore we have $[r_{2n}, a] + [s_{2n}, b] = 0,$ and furthermore we get $[R_\alpha, a] + [S_\alpha, b] = 0.$ This means that $R_\alpha \Wedge a + S_\alpha \Wedge b \in \Ker(\hat{\f} \Wedge \hat{\f} \to \hat{\f})$ and represents an element in homology. The image $\varphi(r_{2n}) = [\varphi(a), {}_{2^n} \varphi(b)] = x^{2n}$ and $\varphi(s_{2n}) = 0,$ then $\hat{\varphi}(R_\alpha) = \sum_{n \in \mathbb N} \alpha_n x^{2n}$ and $S_\alpha \in \Ker(\hat{\varphi}).$ Hence 
    $$(\hat{\varphi} \Wedge \hat{\varphi})\left(R_\alpha \Wedge a + S_\alpha \Wedge b\right) = \left(\sum_{n \in \mathbb N} \alpha_n x^{2n}\right) \wedge 1_{R\lb x \rb}.$$
The isomorphism $H_2(\hat{\L}, R) \cong R\lb x\rb \wedge_{R[x]} R\lb x\rb$ constructed in \Cref{Lamplighter homology} and \Cref{Coinvariants description} sends an element $\sum_i p_i \Wedge q_i$ to $\sum_i p_i \wedge \sigma(q_i).$ Denote by $\widetilde{\varphi}$ the composition of $\hat{\varphi} \Wedge \hat{\varphi}$ with that identification, then the image 
    $$\widetilde{\varphi}(R_\alpha \Wedge a + S_\alpha \Wedge b) = \sum_{n \in \mathbb N} \alpha_n x^{2n} \wedge 1.$$
Then by \Cref{lemma_image_uncount} the image is uncountable. 
\end{proof}

\subsection{Proof of Theorem A}

Throughout the section we denote by $\k$ a field of characteristic $2$.

\begin{lemma}\label{lemmaidentities}
Suppose that $\g$ is a Lie algebra over $\k$ and $a, b, c \in \g.$ Then the following identities hold
$$[[a, {}_{2^n} b], [c, {}_{2^n} b]] = [a, {}_{2^{n+1}} b, c] + [a, {}_{2^n} b, c, {}_{2^n} b]$$
for $n \geqslant 0.$ Hence for any $a, b \in \g$
\begin{align*}
& [a, {}_{2^{n+1}} b, a] = [a, {}_{2^n} b, a, {}_{2^n}b],
& [a, {}_{2^{n+1}} b, a] = [a, b, a, {}_{2^{n+1} - 1} b].
\end{align*}
\end{lemma}
\begin{proof}
The case $n = 0$ follows from the Jacobi identity. Let us denote
$$a_1 = [a, {}_{2^n} b], \ \ \ c_1 = [c, {}_{2^n} b], \ \ \ a_2 = [a, {}_{2^{n+1}} b].$$ By induction hypothesis we have
\begin{align*}
    &[[a, {}_{2^{n+1}} b], [c, {}_{2^{n+1}} b]] = [a_1, {}_{2^{n+1}} b, c_1] + [a_1, {}_{2^n} b, c_1, {}_{2^n} b],\\
    &[a_1, {}_{2^{n+1}} b, c_1] = [a_2, {}_{2^{n+1}} b, c] + [a_2, {}_{2^n} b, c, {}_{2^n} b],\\
    &[a_1, {}_{2^n} b, c_1, {}_{2^n} b] = [a_1, {}_{2^{n+1}} b, c, {}_{2^n} b] + [a_1, {}_{2^n} b, c, {}_{2^{n+1}} b]
\end{align*}
The terms $[a_1, {}_{2^{n+1}}b, c, {}_{2^n}b]$ and $[a_2, {}_{2^{n}} b, c,_{2^n} b]$ are equal, therefore they cancel each other. This completes the induction step. By substituting $c = a$ we obtain 
    $$0 = [[a, {}_{2^k}b], [a, {}_{2^k}b]] = [a, {}_{2^{k+1}} b, a] + [a, {}_{2^k} b, a, {}_{2^k} b],$$
so $[a, {}_{2^{k+1}} b, a] = [a, {}_{2^k} b, a, {}_{2^k} b]$ for all $k \geqslant 0$ and the last part of the statement is obtained by applying that identity $n$ times.
\end{proof}

Next we consider the Lie algebra $\b$ defined in the introduction.

\begin{lemma}\label{H_2(B)}
The second homology group $H_2(\b,\k) = 0.$
\end{lemma}
\begin{proof}
Let $\mathfrak{e} = \f_{\k}(\{a, b\} \cup \{x_i, y_i \mid i \in \mathbb N\})$ and set $\alpha_i = x_i - [a,b,b] - [x_{i + 1},_{2^i} b],$ $\beta_i = y_i - [a, b, a] - [y_{i + 1}, {}_{2^i} b]$ for $i \geqslant 1.$ Then we have the following short exact sequence 
    $$0 \rightarrow \mathfrak{r} \rightarrow \mathfrak{e} \rightarrow \b \rightarrow 0,$$
where $\mathfrak{r}$ is the ideal generated by the set $\{\alpha_i, \beta_i \mid i \in \mathbb N\}.$ Since 
    $$\mathfrak{r} = \sum_{i \in \mathbb N}{[\alpha_i, \mathfrak{e}]} + \sum_{i \in \mathbb N}{[\beta_i, \mathfrak{e}]} + \sum_{i \in \mathbb N}{\k \alpha_i} + \sum_{i \in \mathbb N}{\k \beta_i},$$ 
and $\alpha_i \equiv x_i,$ $\beta_i \equiv y_i \pmod{[\mathfrak{e}, \mathfrak{e}]},$ we have 
    $$\mathfrak{r} \cap [\mathfrak{e}, \mathfrak{e}] = \sum_{i \in \mathbb N}{[\alpha_i, \mathfrak{e}]} + \sum_{i \in \mathbb N}{[\beta_i, \mathfrak{e}]} = [\mathfrak{r}, \mathfrak{e}]$$ 
and the statement follows from the Hopf's formula.
\end{proof}

We say that a homomorphism of Lie algebras $f:\g \to \h $ is {\it 2-connected}, if it induces an isomorphism $H_1(\g,R)\cong H_1(\h,R)$ and an epimorphism $H_2(\g, R)\twoheadrightarrow H_2(\h, R).$ 
The following theorem is a  version of Stallings' theorem \cite[Theorem 3.4]{Stallings} for Lie algebras whose prove is similar to the proof of Stallings.

\begin{theorem}[{Stallings' theorem for Lie algebras, cf. \cite{Stallings}}]\label{Stallings}
Let $R$ be an associative commutative ring and $f: \g \to \h$ be a 2-connected homomorphism of Lie algebras over $R.$ Then $f$ induces isomorphisms
$$\g/\gamma_n(\g) \cong \h/\gamma_n(\h)$$
for any $n \geqslant 1.$
\end{theorem}
\begin{proof} For simplicity in this proof we use the notation $H_i(\g):=H_i(\g,R).$ The proof is by induction. 
For $n=1$ the statement is obvious. For $n=2$ the statement follows from the isomorphism $H_1(\g)=\g/\gamma_2(\g).$ Assume now that $n\geq 3.$ Note that $H_1(\gamma_{n-1}(\g))=\gamma_{n-1}(\g)/[\gamma_{n-1}(\g),\gamma_{n-1}(\g)]$ and 
$$H_0(\g/\gamma_{n-1}, H_1(\gamma_{n-1}(\g)))=H_0(\g, H_1(\gamma_{n-1}(\g)))=\gamma_{n-1}(\g)/\gamma_{n}(\g).$$ By induction hypothesis $f$ induces an isomorphism $\g/\gamma_{n-1}(\g)\cong \h/\gamma_{n-1}(\h).$  The short exact sequence $0\to \gamma_{n-1}(\g)\to \g \to \g/\gamma_{n-1}(\g)\to 0$ induces the following 5-term exact sequence of the Lyndon–Hochschild–Serre spectral sequence
$$H_2(\g) \to H_2(\g/\gamma_{n-1}(\g)) \to  \gamma_{n-1}(\g)/\gamma_n(\g) \to H_1(\g) \to H_1(\g/\gamma_{n-1}(\g)) \to 0.$$ 
If we compare two such five-term exact sequences for $\g$ and $\h$, use the induction hypothesis and the five lemma, we obtain that $f$ induces an isomorphism $\gamma_{n-1}(\g)/\gamma_n(\g)\cong \gamma_{n-1}(\h)/\gamma_n(\h).$ Combing this isomorphism with the isomorphism $ \g/\gamma_{n-1}(\g)\cong \h/\gamma_{n-1}(\h)$ we obtain the isomorphism $ \g/\gamma_{n}(\g)\cong \h/\gamma_{n}(\h).$
\end{proof}

\begin{lemma}
The Lie algebra $\a = \b / \gamma_\omega(\b)$ is parafree.
\end{lemma}
\begin{proof}
Suppose that $\f = \f_{\k}(\{a, b\})$ and consider the map $\varphi: \f \to \b$ that is identical on $a$ and $b.$ Since $\varphi$ induces an isomorphism on $H_1(\f, \k) \to H_1(\b, \k)$ and a surjective map on the second homology group by \Cref{H_2(B)}, it induces isomorphisms $\f/\gamma_n(\f) \xrightarrow{\sim} \b/\gamma_n(\b) = \a / \gamma_n(\a)$ by the \Cref{Stallings}, and so on pronilpotent completions.
\end{proof}

The next lemma is easy to prove by induction.
\begin{lemma}\label{lemmaidentities2}
There are the following identities in $\a/\gamma_{2^n}(\a)$
    $$x_k \equiv [a, b, b] + \sum\limits_{i = k+1}^n [a, {}_{2^i - 2^k + 2} b]  \pmod{\gamma_{2^n}(\a)},$$
    $$y_k \equiv [a, b, a] + \sum\limits_{i = k+1}^n [a, b, a, {}_{2^i - 2^k} b]  \pmod{\gamma_{2^n}(\a)}.$$
Therefore, the images of $x_1$ and $y_1$ under the inclusion $\a \to \hat{\a}$ are the following  
    $$x_1 = [a, b, b] + \sum_{i \geqslant 2}{[a, {}_{2^i}b]}, \ \ \ y_1 = [a, b, a] + \sum_{i \geqslant 2}{[a, b, a, {}_{2^i - 2} b]}.$$
\end{lemma}

Lemmas \ref{lemmaidentities} and \ref{lemmaidentities2} imply that
    $$[x_1, a] + [y_1, b] = 0 \in \hat{\a},$$
and since $\a$ is embedded in $\hat \a,$ the element $x_1 \Wedge a + y_1 \Wedge b$ is in $\Ker(\a \Wedge \a \to \a) \cong H_2(\a, K).$ We are going to show that it is not trivial finishing the proof of Theorem A.

Set $\f = \f_{\k}(\{a, b\})$ and consider the map $\psi: \a \to \hat{\L}_{\k}$ which is the composition of the embedding $\a \to \hat{\a} \cong \hat{\f}$ and the map $\hat{\varphi}: \hat{\f} \to \hat{\L}_{\k}$ from \Cref{image}. Note that for the induced map $\hat{\psi}: \hat{\a} \to \hat{\L}_{\k}$
    $$\hat{\psi}(x_1) = x^2 + \sum_{i \geqslant 2} x^{2i}, \ \ \ \hat{\psi}(y_1) = 0.$$
If we identify $H_2(\hat{\L}_{\k}, {\k})$ with $\k\lb x\rb \wedge_{\k[x]} \k\lb x\rb$ as before, then the image of $x_1 \Wedge a + y_1 \Wedge b$ is $\left(\sum\limits_{i = 1}^{\infty} x^{2^i}\right) \wedge 1.$ By \Cref{lemma_not_rational} $\sum\limits_{i = 1}^{\infty} x^{2^i}$ is not rational, and hence, the image is not trivial by \Cref{rational}.

\subsection{Proof of Theorem B} We follow the notation of \Cref{image} for $R = \mathbb Z.$
Let $\f = \f_{\mathbb Z}(a,b).$  As in the proof of Theorem \ref{image}, we consider the following elements
    $$r_{2^n} = [a, {}_{2^n}b], \ \ \ s_{2^n} =  \left(\sum\limits^{2^{n-1}}_{i=1} (-1)^i[[a, {}_{2^n-i}b], [a, {}_{i-1}b]] \right) \in \f$$
    $$t_{2^n} = r_{2^n} \wedge a + s_{2^n} \wedge b \in \f \wedge \f$$
for $n=1,2,\dots$
Then the sum $\sum_n 2^nt_{2^n}$ defines a cycle in
    $$\hat \f\wedge \hat \f\wedge \hat \f\longrightarrow \hat \f\wedge \hat \f\buildrel{[\ ,\ ]}\over\longrightarrow \hat \f,$$
and represents an element in homology by \Cref{Ellis description}.  The power series $\sum_n 2^n x^{2^n}$ is not rational by \Cref{lemma_not_rational}. 
Then the image of the sum $\sum_n 2^nt_{2^n}$
    $$\widetilde{\varphi}\left(\sum\limits_{i = 1}^{\infty} 2^nt_{2^n}\right) = \left(\sum\limits_{i = 1}^{\infty} 2^nx^{2^n}\right) \wedge 1$$
is not trivial by \Cref{rational}, where $\widetilde{\varphi}$ was defined in the proof of \Cref{image}. Hence,
$\sum_n 2^nt_{2^n}$ defines a nonzero element in $H_2(\hat \f, \mathbb Z)$. 

Now observe that $\sum_n 2^nt_{2^n}$ is 2-divisible. Indeed, the finite sum
$
\sum_{n=1}^k 2^nt_{2^n}
$
defines an element from the kernel of the map
$ [\ ,\ ]: \f \Wedge \f\longrightarrow \f$
and since the second homology of the free Lie algebra are trivial, $\sum_{n=1}^k 2^nt_{2^n}\in \f \Wedge \f$ is trivial. Therefore we have $\sum 2^nt_{2^n}=2^k\sum_{n\geq k+1} 2^{n-k}t_{2^n}$ in $H_2(\hat{\mathfrak{f}},\mathbb Z)$.  Since the element $\sum_{n\geq k+1} 2^{n-k}t_{2^n}$ also defines a nonzero element in homology, we conclude that the element $\sum 2^nt_{2^n}$ is $2^k$-divisible for every $k$. That is, we constructed a 2-divisible element in $H_2(\hat \f, \mathbb Z)$.

Finally we prove that cohomological dimension of $\hat \f$ is at least $3.$
Assume the contrary, that ${\rm cd}(\hat \f)\leq 2.$
Consider a projective resolution $P_\bullet\twoheadrightarrow \mathbb Z$ over the enveloping algebra $U(\hat \f).$
Set $\Omega^n={\rm  Coker}(P_{n+1}\to P_n)$ for $n\geq 0.$
Then the long exact sequence of the short exact sequence $\Omega^{n+1}\rightarrowtail P_n \twoheadrightarrow \Omega^n$ imply that $${\rm Ext}^{m}(\Omega^{n+1}, \:\cdot\: )={\rm Ext}^{m+1}(\Omega^n ,\: \cdot\: ), \hspace{1cm} H_m(\hat \f,\Omega^{n+1})=H_{m+1}( \hat \f, \Omega^n) $$
for $m\geq 1$
and there is a monomorphism $$ H_1(\hat \f,\Omega^n) \rightarrowtail H_0(\hat \f, \Omega^{n+1}).$$
 Therefore ${\rm Ext}^{1}(\Omega^2 , \:\cdot\: )={\rm Ext}^{3}(\mathbb Z,\: \cdot\: )=H^3(\hat \f,\:\cdot\:)=0.$ It follows that $\Omega^2$ is projective.
On the other hand we have a monomorphism $H_2(\hat \f,\Z)=H_1(\hat \f,\Omega^1 ) \hookrightarrow H_0(\hat \f, \Omega^2).$ Since $\Omega^2$ is projective, $H_0(\hat \f, \Omega^2)$ is a free abelian group, and hence $H_2(\hat \f, \mathbb Z)$ is a free abelian group. This contradicts to the fact that  $H_2(\hat \f, \mathbb Z)$  has a nontrivial $2$-divisible element.

\subsection{Proof of Proposition 1.}

In order to prove this proposition we need to recall some statements from the theory of $H\Z$-localization that can be found in \cite{OrrShelahFarjoun} and \cite{B1}. During this section for a group $G$ we denote by $H_2(G) = H_2(G, \mathbb Z).$

Let $G$ be a group, $X$ be a set that we call the set of variables and $F=F(X)$ be the free group generated by $X$. An element $w$ of the free product $G*F$ is called monomial with coefficients in $G.$ A monomial $w$ is called {\it acyclic}, if its image in $F_{ab}$ is trivial via the composition of the maps $G*F \to F \to F_{ab},$
where the first map sends $G$ to $1$ and the second one is the canonical projection. Let $\mathcal{S}=(w_x)_{x\in X}$ be a family of acyclic monomials indexed by $X.$  A $\Gamma$-system of equations defined by $\mathcal{S}$ is the family of equations $(x=w_x)_{x\in X}.$ A solution of such a system is a map $X\to G$ such that $xw_x^{-1}$ is in the kernel of the induced map $G*F\to G.$

A homomorphism $f:G\to G'$ is called $2$-connected if it induces an isomorphisms $H_1(G)\cong H_1(G')$ and an epimorphism $H_2(G)\twoheadrightarrow H_2(G').$

A group $G$ is $H{\mathbb Z}$-local if for every 2-connected homomorphism $\varphi: A \to B$
and every homomorphism of groups $f : A \to G$ there is a group homomorphism $g : B \to G$ such that $g \circ \varphi = f.$

\vbox{
\begin{lemma} Let  $\mathcal S=(w_x)_{x\in X}$ be an $X$-indexed family of acyclic monomials in $G*F$. Consider the group $G_\mathcal{S}=(G*F)/R,$ where $R$ is the normal subgroup generated by the elements $xw_x^{-1}.$ Then the map
$$G\longrightarrow  G_\mathcal{S}$$
is $2$-connected.
\end{lemma}
}
\begin{proof}
The fact that the induced map $H_1(G)\to H_1(G_\mathcal{S})$ is an isomorphism, is obvious.

Prove that $H_2(G)\to H_2(G_\mathcal{S})$ is an epimorphism. It is easy to see that $H_2(G)=H_2(G*F).$ Indeed it follows from the long exact sequence in homology for free products with amalgamation (see \cite{Brown}). Hence we need to prove that $H_2(G*F)\to H_2(G_\mathcal{S})$ is an epimorphism. Recall that for any group $A$ and any its normal subgroup $U$ the cokernel of $H_2(A)\to H_2(A/U)$ is isomorphic to $(U\cap [A,A])/[U,A].$ Therefore we need to prove that $(R\cap  [G*F,G*F])/[R,G*F]=0.$ Let us write elements of $R/[R,G*F]=(R_{ab})_{G*F}$ and $F_{ab}$ in the additive notation. Then any element of $R/[R,G*F] $ can be presented as a linear combination $\theta= \sum \alpha_x (xw_x^{-1}).$ The image of $\theta$ in $F_{ab}$ is $\sum \alpha_xx.$ So $\theta$ in $(R\cap  [G*F,G*F])/[R,G*F]$ only if $\sum \alpha_xx=0,$ and hence $\alpha_x=0$ for any $x.$  Then $(R\cap  [G*F,G*F])/[R,G*F]=0.$
\end{proof}

\begin{lemma}\label{lemma_any_element_of_hatG} Let $G$ be a finitely generated group. Then any element of $\hat G$ is an element of a solution of a countable $\Gamma$-system of equations with coefficients in ${\rm Im}(G\to \hat G).$
\end{lemma}
\begin{proof}
It is proved in \cite{OrrShelahFarjoun} that a group is $H\Z$-local if and only if any $\Gamma$ system of equations has a unique solution. Moreover, they prove that it is enough to consider countable $\Gamma$-systems of equations: a group $L$ is $H\Z$-local if and only any countable $\Gamma$-system of equations has a unique solution.

If $G\subseteq H,$ then the set of all solutions of $\Gamma$-systems of equations with constants in $G$ is called $\Gamma$-closure of $G$ in $H.$ It is proved in \cite{OrrShelahFarjoun} that $\Gamma$-closure of $G$ equals to Bousfiled's $H\Z$-closure of $G$ in $H.$ Again, one can check that it is enough to consider countable $\Gamma$-systems of equations. In particular, this means that any element of $H\Z$-localization $LG$ is an element of a solution of a countable $\Gamma$-system of equations with coefficients in the image of $G$. Then the assertion follows from the fact $\hat G=LG/\gamma_\omega(LG).$
\end{proof}

\begin{proposition} Let $G$ be a finitely generated group. Then for any countable subset $A\subseteq \hat G$ there exists a subgroup $H\subseteq \hat G$ such that ${\rm Im}(G\to \hat G) \cup A \subseteq H$ and the induced maps $G/\gamma_n(G)\to H/\gamma_n(H)$ are isomorphisms.
\end{proposition}
\begin{proof}  For simplicity we set $G/\gamma_n=G/\gamma_n(G)$ for any group $G.$
Lemma \ref{lemma_any_element_of_hatG} implies that any element $a\in A$ is an element of a solution of a countable $\Gamma$-system of equations. A countable union of countable $\Gamma$-systems of equations is a $\Gamma$-system of equations. Therefore there exists a countable family of acyclic monomials $\mathcal{S}=(w_i)_{i=1}^\infty$ from $G*F(x_1,x_2,\dots)$ such that $A$ lies in the image of the solution $\{x_1,x_2,\dots\}\to \hat G$ of the $\Gamma$-system of equations $(x_i=w_i).$ Consider the group
$$G_\mathcal{S}=(G*F(x_1,x_2,\dots))/R,$$
where $R$ is the normal subgroup generated by the elements $x_iw_{i}^{-1}.$
Since the map $G\to G_\mathcal{S}$ is $2$-connected,  then by Stallings' theorem we have an isomorphism $G/\gamma_n \cong  G_\mathcal{S}/\gamma_n.$ In particular $\hat G\cong \hat G_\mathcal{S}.$  Therefore we obtain a map $f: G_\mathcal{S}\to \hat G,$ whose kernel is $\gamma_\omega(G_\mathcal{S}).$ Denote by $H$ the image of $f.$ Then $H\cong G_\mathcal{S}/\gamma_\omega$ and $H/\gamma_n\cong G_\mathcal{S}/\gamma_n \cong G/\gamma_n.$ The restriction $f|_{\{x_1,x_2\dots\}}$ is a solution of the $\Gamma$-system of equations $(x_i=w_i)$. Since $\hat G$ is $H\Z$-local, any $\Gamma$-system of equations has a unique solution. Therefore $A\subseteq H.$
\end{proof}
\begin{corollary} Let $F$ be a finitely generated free group. Then for any countable subset $A\subset \hat F$ there exists a countable parafree subgroup $G\subseteq \hat F$ such that $A\subseteq G.$
\end{corollary}

The Proposition 1 follows from this corollary.

\end{document}